\documentstyle[12pt]{article}

\begin{document}

\newtheorem{theorem}{Theorem}[section]
\newtheorem{corollary}[theorem]{Corollary}
\newtheorem{definition}[theorem]{Definition}
\newtheorem{conjecture}[theorem]{Conjecture}
\newtheorem{question}[theorem]{Question}
\newtheorem{lemma}[theorem]{Lemma}
\newtheorem{proposition}[theorem]{Proposition}
\newtheorem{example}[theorem]{Example}
\newtheorem{problem}[theorem]{Problem}
\newenvironment{proof}{\noindent {\bf
Proof.}}{\rule{3mm}{3mm}\par\medskip}
\newcommand{\remark}{\medskip\par\noindent {\bf Remark.~~}}
\newcommand{\pp}{{\it p.}}
\newcommand{\de}{\em}

\newcommand{\JEC}{{\it Europ. J. Combinatorics},  }
\newcommand{\JCTB}{{\it J. Combin. Theory Ser. B.}, }
\newcommand{\JCT}{{\it J. Combin. Theory}, }
\newcommand{\JGT}{{\it J. Graph Theory}, }
\newcommand{\ComHung}{{\it Combinatorica}, }
\newcommand{\DM}{{\it Discrete Math.}, }
\newcommand{\ARS}{{\it Ars Combin.}, }
\newcommand{\SIAMDM}{{\it SIAM J. Discrete Math.}, }
\newcommand{\SIAMADM}{{\it SIAM J. Algebraic Discrete Methods}, }
\newcommand{\SIAMC}{{\it SIAM J. Comput.}, }
\newcommand{\ConAMS}{{\it Contemp. Math. AMS}, }
\newcommand{\TransAMS}{{\it Trans. Amer. Math. Soc.}, }
\newcommand{\AnDM}{{\it Ann. Discrete Math.}, }
\newcommand{\NBS}{{\it J. Res. Nat. Bur. Standards} {\rm B}, }
\newcommand{\ConNum}{{\it Congr. Numer.}, }
\newcommand{\CJM}{{\it Canad. J. Math.}, }
\newcommand{\JLMS}{{\it J. London Math. Soc.}, }
\newcommand{\PLMS}{{\it Proc. London Math. Soc.}, }
\newcommand{\PAMS}{{\it Proc. Amer. Math. Soc.}, }
\newcommand{\JCMCC}{{\it J. Combin. Math. Combin. Comput.}, }
\newcommand{\GC}{{\it Graphs Combin.}, }

\title{The Laplacian eigenvalues of graphs: a survey\thanks{
Supported by National Natural Science Foundation of China (No.
10531070), National Basic Research Program of China 973 Program
(No. 2006CB805901), National Research Program of China 863 Program
(No. 2006AA11Z209) and the Natural Science Foundation of Shanghai
(Grant No. 06ZR14049).}}
\author{  Xiao-Dong Zhang \\
{\small Department of Mathematics}\\
{\small Shanghai Jiao Tong University} \\
{\small  800 Dongchuan road, Shanghai, 200240, P.R. China}\\
{\small Email:  xiaodong@sjtu.edu.cn}
 }
\date{}
\maketitle
 \begin{abstract}
  The Laplacian matrix of a simple graph is the difference of the diagonal matrix of vertex
degree and the (0,1) adjacency matrix.  In the past decades, the
Laplacian spectrum has received much more and more attention,
since it has been applied to several fields, such as randomized
algorithms, combinatorial optimization problems and machine
learning.
 This paper is primarily a survey of various aspects of the
 eigenvalues   of the Laplacian matrix of a graph  for the past teens.
In addition, some new unpublished  results and questions are
concluded. Emphasis is given on classifications of the upper and
lower bounds for the Laplacian eigenvalues of graphs (including
some special graphs, such as trees, bipartite graphs,
triangular-free graphs, cubic graphs,  etc.) as a function of
other graph invariants, such as degree sequence, the average
2-degree, diameter, the maximal independence number, the maximal
matching number,   vertex connectivity, the domination number, the
number of the spanning trees, etc.

 \end{abstract}


\vskip 0.5cm

\section{Introduction}

The Laplacian matrix has a long history. The first  celebrated
result is  attributable to Kirchhoff \cite{kirchhoff1847} in an
1847 paper concerned with electrical networks. However, it did not
receive much attention until  the work of Fiedler, which appeared
in 1973 \cite{fiedler1973} and 1975 \cite{fiedler1975}.  Mohar in
his survey \cite{mohar1991}  argued that, because of its
importance in various physical and chemical theories, the spectrum
of the Laplacian matrix  is more natural and important than the
more widely studied adjacency spectrum. In \cite{alon1986}, Alon
used the smallest positive eigenvalue of the Laplacian matrix  to
estimate the expander and magnifying coefficients of graphs.

 There are several books and survey papers concerned with the
 Laplacian matrix of a graph. For example,  in 1997, Chung \cite{Chung1998} published his book entitled
 "Spectral graph theory" which investigated the theory of the Laplacian
 matrix with aid of the ideas and methods  of differential
 manifold. In 1991 and 1992, Mohar \cite{mohar1991}, \cite{mohar1992} surveyed
a detailed introduction to the Laplacian matrix. Further, in 1997,
he surveyed  several applications of eigenvalues of the Laplacian
matrices of graphs in graph theory and in combinatorial
optimization.  In 1994, Merris \cite{merris1994} surveyed the
properties of the Laplacian matrix from the view of linear algebra
and graph theory. Further, in 1995, he \cite{merris1995} surveyed
the relations between the parameters and the spectrum of the
Laplacian matrix and some applications which was not appeared in
\cite{merris1994}. In 1991, Grone \cite{grone1991} surveyed the
geometry properties of the Laplacian matrix. Recently, Abreu
\cite{abreu2007} surveyed the old and new results of the second
smallest Laplacian eigenvalue. For the more background and
motivation on research of the Laplacian matrix, the reader may be
referred to the above books, surveys and their references in
there.

This paper is  a survey of recent new results and questions on the
spectrum of the Laplacian matrix.  The present content is biased
by the viewpoint and the interests of the authors and can not be
complete. Therefore we apologize to all those who feel that their
work is missing in the references or has not been emphasized
sufficiently in this survey.

Let $G = (V,~E)$ be a simple graph (no loops or multiple edges)
with vertex set $V(G)=\{v_1,\cdots, v_n\}$ and edge set $E(G)$.
Denote by $d(v_i)$ or $d_G(v_i)$ the {\it degree} of vertex $v_i$.
If $D(G)=diag(d(u), u\in V)$ is the diagonal matrix of vertex
degrees of $G$ and $A(G)$ is the $(0,1)$ {\it  adjacency matrix}
of $G$, then the matrix $L(G)=D(G)-A(G)$ is called the {\it
Laplacian matrix} of a graph $G$. It is obvious that $L(G)$ is
positive semidefinite and singular $M-$matrix. Thus the all
eigenvalues of $L(G)$  are called the  {\it Laplacian eigenvalues}
(or sometimes just eigenvalues) of $G$  and arranged in
nonincreasing order:
$$\lambda_1\ge \lambda_2\ge\cdots\ge \lambda_{n-1}\ge
\lambda_n=0.$$ When more than one graph is under discussion, we
may write $\lambda_i(G)$ instead of $\lambda_i$. From the
matrix-tree theorem, $\lambda_{n-1}>0$ if and only if $G$ is
connected. This observation led Fiedler to define the {\it
algebraic connectivity} of $G$ by $\alpha(G)=\lambda_{n-1}(G)$,
which may be considered a quantitative measure of connectivity.

 Let $G=(V, E)$ be a  simple graph
with vertex set $V(G)=\{v_1,\cdots, v_n\}$ and edge set
$E(G)=(e_1, e_2,\cdots, e_m)$. For each edge $e_k=(v_i, v_j)$,
choose one of $v_i$ or $v_j$  to be the positive end of $e_k$ and
the other to be the negative end. We refer to this procedure by
saying $G$ has been given an {\it orientation}. For an arbitrary
given orientation of $G$, the {\it oriented vertex-edge incidence
matrix}  is the $n\times m$ matrix $Q=Q(G)=(q_{ij})$, where
$$q_{ij}=\left\{\begin{array}{ll}
+1,  & {\rm if}\  v_i\ {\rm is\  the\  positive\  end\  of}\  e_j\\
-1,  & {\rm if}\  v_i \ {\rm is\  the\ negative\ end\ of\ }\ e_j\\
0, & {\rm otherwise.} \end{array}\right.$$
 While $Q$ depends on the orientation of $G$, $QQ^T$ does not. In
 fact, for any orientation of $G$, it is easy to see that
 $$Q(G)Q(G)^T=D(G)-A(G)=L(G).$$
Thus one may also describe $L(G)$ by means of its quadratic form
$$x^TL(G)x=(Q(G)^Tx)^T(Q(G)^Tx)=\sum(x_i-x_j)^2,$$
where $x=(x_1,\cdots, x_n)^T$ is $n-$dimension real vector and the
sum is taken over all pairs $i<j$ for which $(v_i, v_j)\in E(G)$.

 The first  appearance of $L(G)$  may occur in Kirchhoff's {\it matrix-tree
 theorem } \cite{kirchhoff1847}:
 \begin{theorem}(\cite{kirchhoff1847})\label{kirchhoff1847}
 Let $L(i|j)$ be the $(n-1)\times (n-1)$ submatrix of $L(G)$, which
 is obtained by deleting its $i-$th row and $j-$column. Denote by $\tau(G)$ the
 number of  spanning trees in $G$. Then
 $$\tau(G)=(-1)^{i+j}\det L(i|j)=\frac{1}{n}\ \prod_{i=1}^{n-1}\lambda_i.$$
  \end{theorem}

  In view of this result, $L(G)$ is sometimes called the {\it
  Kirchhoff matrix} or {\it matrix of admittance}
  (admitance=conductivity, the reciprocal of impedance).
  However, we will refer to $L(G)$ as a Laplacian matrix because it is a
discrete analogue of the Laplace differential operator. The
Laplacian matrix of  a graph and its eigenvalues can be used in
several areas of mathematical research and  have a physical
interpretation in various physical and chemical theories. The
adjacency matrix of a graph and its eigenvalues have been much
investigated in the monographs \cite{Cvetkovic1980} and
\cite{cvetkovic1988}. The normalized Laplacian matrix
${\mathcal{L}}(G)=D^{-1/2}L(G)D^{-1/2}$ of a graph and its
eigenvalues has studied in the monographs \cite{Chung1998}.

In this paper, we survey the Laplacian eigenvalues of a graph. In
section 2, some basic and important properties of the Laplacian
eigenvalues are reviewed. In section 3, the largest Laplacian
eigenvalue is heavily investigated. Many upper and lower bounds
for the largest Laplacian eigenvalues  of graphs  and special
graphs (including tree, cubic graphs, triangular free graphs,
etc.) are presented. Proofs  of part important results are also
given. In section 4, the second Laplacian eigenvalue is studied
and a question is proposed. In section 5, the bounds for the
$k-$largest Laplacian eigenvalue are discussed. In section 6, the
upper and lower bounds for  the second smallest Laplacian
eigenvalue, i.e., algebraic connectivity, are studied. Moreover,
the relations between algebraic connectivity and graph parameters
are obtained. In section 7, the sum of the   Laplacian eigenvalues
are investigated with emphasizing on two conjectures of Grone and
Merris in \cite{grone1994}.

\section{Preliminary}
 Let $G=(V(G), E(G))$
be a simple graph. The line graph of $G$, written $G^l$, is the
graph whose vertex set is the edge set $E(G)$ of $G$ and  whose
two vertices  are adjacent if and only if they have one common
vertex  in $G$. Denoted by $D(G)=diag(d(u), u\in V)$  and $A(G)$
the diagonal matrix of vertex degrees of $G$ and  the $(0,1)$ {\it
adjacency matrix} of $G$ respectively. The matrix $K(G)=D(G)+A(G)$
is called the {\it unoriented Laplacian matrix} of $G$.  Moreover,
denote by $Q(G)$ the oriented vertex-edge incidence matrix. Let
$X=(x_{ij})$ be an $(n\times n)$ matrix. Denote by
$|X|=(|x_{ij}|)$ the matrix whose entries are absolute values of
the entries of $X$.  Denote by $\rho(X)$ the largest modulus of
eigenvalues of $X$. Then we sum up  some preliminary results from
\cite{li1997}, \cite{merris1991}, \cite{merris1994} \cite{shu2002}
as follows:
\begin{lemma}\label{graph-line}
Let $G$ be a simple graph. Then
\begin{equation}\label{unori-line}
K(G)=D(G)+A(G)=|Q(G)Q(G)^T|=|Q(G)| |Q(G)^T|.
\end{equation}
\begin{equation}\label{graph-adj}
|Q^T(G)Q(G)|=2I+A(G^l),
\end{equation}
where $I$ is the identity matrix.
\begin{equation}\label{graph-line-eig}
\lambda_1(G)\le \rho(K(G))=2+\rho(A(G^l))
\end{equation}
with equality if and only $G$ is bipartite.
\end{lemma}
 A {\it semiregular graph} $G=(V,E)$ is a graph  with bipartition $(V_1,
 V_2)$ of $V$ such that all vertices in $V_i$ have the same degree
 $k_i$ for $i=1,2$.

\begin{lemma}(\cite{zhang2002a}\label{graph-line-equal}
Let $G$ be a simple connected graph. Then  the line graph $G^l$ of
$G$ is regular  or semiregular if and only if $G$ is regular or
semiregular or a path of order $4$.\end{lemma}
\begin{proof}
since sufficiency is obvious, we only consider necessity. If $G^l$
is k-regular, then for each edge $e_{uv}=(u,v)\in E(G)$, the
degree of vertex $e_{uv}$ in $G^l$ is equal to
$d_{G^l}(e_{uv})=d_G(u)+d(v)-2$.  Hence if two vertices of $G$
share a common vertex, then they have the same degree. Since $G$
is connected, this implies that there are at most two different
degrees. If two adjacent vertices have same degree, it is easy to
show that $G$ is regular by means of induction argument. If $G$
contains a cycle of odd length, then it must have two adjacent
vertices  with the same degree.  Therefore, if $G$ is not regular,
then it does not contain any cycle of odd length, which implies
that $G$ is bipartite. So $G$ is semiregular.
\end{proof}

\begin{lemma}(\cite{merris1994})\label{graph-com-spe}
Let $G$ be a simple graph on $n$ vertices and $G^c$ be the
complement graph of $G$ in the complement graph. Then
\begin{equation}\label{graph-com-sec-1}
\lambda_1(G)\le n .\end{equation} \
 \begin{equation}\label{graph-com-2}
 \lambda_i(G^c)=n-\lambda_{n-i}(G)\  {\rm for} \ i=1,\cdots, n-1.
 \end{equation}
 \end{lemma}
 \begin{proof} Since
 $$L(G)+L(G^c)=nI-J,$$
where $J$ is the $n\times n$ matrix each of whose entries is 1. It
follows that the Laplacian spectrum of $G^c$ is
$$n-\lambda_{n-1}(G)\ge n-\lambda_{n-2}(G)\ge\cdots\ge
n-\lambda_1(G)\ge 0.$$ Therefore the assertion holds.
\end{proof}

There are several useful min-max formulas for the expression of
eigenvalues of a symmetric matrix and their sums. If $M$ is a real
symmetric matrix of order $n\times  n$ and $ \mathbf{R}^n$ is the
$n$ real dimension vector space,  then  Rayleigh-Ritz  ration (see
p.176 in \cite{horn1985}) may be expressed as follows.

\begin{equation}\label{max-eig}
\lambda_1(M)=\max\{x^TMx\ \ | \ \parallel x\parallel =1, x\in
 \mathbf{R}^n \ \ \}
 \end{equation}
and
\begin{equation}\label{min-eig} \lambda_1(M)=\min\{x^TMx\ \ |
\
 \parallel x\parallel =1, x\in
 \mathbf{R}^n \ \ \}.
 \end{equation}
In general, the min-max characterization of $\lambda_k(M)$ is
called {\it Courant-Fischer "min-max theorem"}(see p.179 in
\cite{horn1985})
\begin{equation}\label{courant-fischer-th}
\lambda_k(M) = \max_{ U}\min_{ x}\{x^TMx\ | \ \parallel x\parallel
=1, x\in U\ \},
\end{equation}
 where the first minimum is over all
$k-$dimensional subspaces $U$ of $R^n$.

\section{The Largest Laplacian eigenvalue}
In this section, we  will discuss the upper and lower bounds for
the largest Laplacian eigenvalue for graphs and several kinds of
special graphs, including trees, triangular-free graphs, cubic
graphs. There are a lot of papers focus on this topic.

 \subsection{The upper bound versus degree sequences}

In 1985, Anderson and Morley  \cite{anderson1985} may first obtain
the upper bound for the largest Laplacian eigenvalue. They  showed
the following:

\begin{theorem}(\cite{anderson1985}) \label{anderson1985-th}
Let $G$ be a simple graph. Then
\begin{equation}\label{anderson1985-1}
 \lambda_1\le
\max\{d(u)+d(v)|(u,v)\in E(G)\},
\end{equation}
where $d(u)$ is the degree of vertex $u$. \end{theorem}

In 1997, this result  was improved by Li and Zhang \cite{li1997}.
Their  main result is as follows:

\begin{theorem}(\cite{li1997})\label{li1997-th}Let $G$ be a simple graph. Denote
by $r=\max\{d(u)+d(v)|(u,v)\in E(G)\}$ and $s=\max\{d(u)+d(v)|
(u,v)\in E(G)-(x,y)\}$ with $(x,y)\in E(G)$ such that
$d(x)+d(y)=r$. Then
\begin{equation}\label{li1997-1}
\lambda(G)\le 2+\sqrt{(r-2)(s-2)},
\end{equation}
\end{theorem}

Pan in \cite{pan2002} gave the necessary and sufficient conditions
for the holding of  equality in (\ref{li1997-1}) .  In fact this
result may further be improved. We can state as follows:

\begin{theorem}\label{li-zhang-new}Let $G$ be a simple  connected graph.
Then
\begin{equation}\label{lizhang-new-1}
\lambda(G)\le 2+\max\left\{\sqrt{(d(u)+d(v)-2)(d(u)+d(w)-2)}\
\right\},
\end{equation}
where the maximum is taken over all pairs $(u,v), (u,w)\in E(G)$.
Moreover, equality holds in (\ref{lizhang-new-1}) if and only if
$G$ is  regular bipartite graph or a semiregular graph, or a path
of order four.
\end{theorem}

\begin{proof} For each edge $e_{uv}=(u,v)\in E(G)$,  the
degree $d(e_{uv})$ of vertex $e_{uv}$ in $G^l$ is equal to
$d_G(u)+d_G(v)-2$. By  Lemma 2.1 in \cite{berman2001},
$$\rho(G^l)\le\max\left\{\sqrt{(d(u)+d(v)-2)(d(u)+d(w)-2)}\ \right\},$$
where the maximum is taken over all pairs $(u,v), (u,w)\in E(G)$.
Hence it follows from (\ref{graph-line-eig}) in
Lemma~\ref{graph-line} that  (\ref{lizhang-new-1}) holds. Clearly,
if $G$ is regular bipartite graph or a semiregular graph, or a
path of order four, by some calculations, it is easy to argue that
equality in (\ref{lizhang-new-1}) holds. Conversely, if equality
in (\ref{lizhang-new-1}) holds, then
$$\rho(G^l)=\max\left\{\sqrt{(d(u)+d(v)-2)(d(u)+d(w)-2)}\ \right\},$$
where the maximum is taken over all pairs $(u,v), (u,w)\in E(G)$.
By Lemma 2.1 in \cite{berman2001}, $G^l$ is regular  or
semiregular.  Consequently it follows from
Lemma~\ref{graph-line-equal} that $G$ is  regular bipartite graph
or a semiregular graph, or a path of order four.
\end{proof}
  We notice that Theorem~\ref{li-zhang-new} is  a new result and better than
 Theorems~\ref{anderson1985-th} and \ref{li1997-th}.   In 2002,   Shu, Hong and Wen \cite{shu2002} gave
an upper bound in terms of
 degree sequences.
 \begin{theorem}(\cite{shu2002})\label{shu2002-th}
 Let $G$ be a simple graph.
 Assume that the degree sequence of $G$ is
 $d_1\ge d_2\ge \cdots\ge d_n$. Then
 \begin{equation}\label{shu2002-1}
\lambda_1(G)\le
d_n+\frac{1}{2}+\sqrt{(d_n-\frac{1}{2})^{2}+\sum_{i=1}^{n}d_i(d_i-d_n)}
\end{equation}
with equality  if and only if $G$ is a regular bipartite graph.
\end{theorem}

{\bf Sketch of} \begin{proof} With the aid of the result in
\cite{hong2001} and Lemma~\ref{graph-line}, it not difficult to
argue with some calculations that (\ref{shu2002-1}) holds.
\end{proof}

 Das in \cite{das2005} also  gave  several related upper bounds
 for the largest Laplacian eigenvalue in terms of degree sequence.

\subsection{The upper bounds
versus  the average 2-degree}

Let $G$ be a simple graph. Denote by $m(v)$ the average of the
degrees of the vertices adjacent to $v$.  Then $d(v)m(v)$ is the
$"2-degree"$ of vertex $v$.  In 1998, Merris \cite{Merris1998a}
used another approach method to provide another upper bound:

\begin{theorem}(\cite{Merris1998a})\label{merris1998-th}
Let $G$ be a simple graph. Then
\begin{equation}\label{merris1998-1}
\lambda_1(G)\le \max\{ d(v)+m(v)\ |\ v\in V(G)\}.
\end{equation}
\end{theorem}

 We observed that Merris' bound (\ref{merris1998-1}) was only involved in
one vertex, while Li and Zhang's bound (\ref{li1997-1}) was
involved in the adjacent vertices. It was natural to stimulate us
to consider whether there was  an better upper bound  than Merris'
upper bound for graphs with the adjacent relations.   Li and Zhang
in \cite{li1998} followed this idea and obtained an better upper
bound. Later Pan in \cite{pan2002} characterized equality
situation.

\begin{theorem}(\cite{li1998},\cite{pan2002})\label{lizhang1998-th}
Let $G$ be a simple graph. Then
\begin{equation}\label{li1998-1}
\lambda_1(G)\le \max\left\{
\frac{d(u)(d(u)+m(u))+d(v)(d(v)+m(v))}{d(u)+d(v)}\ :\
 (u,v)\in E(G)\ \right\}.
  \end{equation}
  If $G$ is connected, then equality in (\ref{li1998-1}) holds
  if and only if $G$ is regular bipartite or semiregular.
 \end{theorem}
{\bf Sketch of }\begin{proof} Let $P$ be sum of the degree
diagonal matrix of the line graph $G^l$ of  a graph $G$ and two
multiple of the identity matrix.   Let
$$N=P^{-1}(2I+A(G^l))P^{-1}.$$
If $e_{uv}=(u,v)$ is an edge of $G$, then $e_{uv}$ is an vertex of
$G^l$ and the corresponding row sum of $N$ is equal to
$$\frac{\sum_{x\sim u}(d(x)+d(u))+\sum_{y\sim v}(d(y)+d(v))}{d(u)+d(v)}
=\frac{d(u)(d(u)+m(u))+d(v)(d(v)+m(v))}{d(u)+d(v)},$$ where $u\sim
 v$ mean that $u$ and $v$ in $G$ are adjacent. Hence
$$\rho(N)\le \max\left\{
\frac{d(u)(d(u)+m(u))+d(v)(d(v)+m(v))}{d(u)+d(v)}\ :\
 (u,v)\in E(G)\ \right\}.$$
 On the other hand,
 by Lemma~\ref{graph-line}, we have
 $$\lambda_1(G)\le \rho(2I+A(G^l))=\rho(N).$$
 Therefore (\ref{li1998-1}) holds. For the equality situation, the proof is omitted.
\end{proof}
Denote by $t=\max\{ d(v)+m(v)\ |\ v\in V(G)\}$. It is obvious that
(\ref{li1998-1}) is better than (\ref{merris1998-1}),  since
$$
\max\left\{ \frac{d(u)(d(u)+m(u))+d(v)(d(v)+m(v))}{d(u)+d(v)} :
 (u,v)\in E(G)\right\}\le \frac{d(u)t+d(v)t}{d(u)+d(v)}=t.$$
By a similar method, we could get another two upper bounds
\begin{theorem}(\cite{zhang2004})\label{zhang2004-th} Let $G$ be
a simple connected graph. Denote by $t(u)=d(u)+m(u)$. Then
\begin{equation}\label{zhang2004-1}
\lambda_1(G)\le\max\left\{2+\sqrt{(d(u)(t(u)-4) + d(v)(t(v)-4)+4}\
\right\}
\end{equation}
and
\begin{equation}\label{zhang2004-2}
\lambda_1(G)\le\max\left\{\sqrt{d(u)t(u) + d(v)t(v)}\ \right\},
\end{equation}
 where the maximum is taken over all pairs $(u,v)\in E(G).$
Moreover,  equality  in (\ref{zhang2004-1})  holds if and only if
$G $ is bipartite regular or semi-regular, or a path of order
four. Equality in (\ref{zhang2004-2}) holds  if and only if $G $
is bipartite regular or semi-regular.
\end{theorem}

\subsection{The upper bound versus eigenvectors}

 In this subsection, we use  the relationships between eigenvalues and eigenvectors to investigate the
 largest Laplacian eigenvalue.  Li and Pan in \cite{li2001} showed
 the following result.
\begin{theorem}(\cite{li2001})\label{li2001-th}
Let $G$ be a simple connected graph. Then
\begin{equation}\label{li2001-1}
\lambda_1(G)\le \max\{ \sqrt{2d(u)(d(u)+m(u))}\  |\  u\in V(G)\}
\end{equation}
with equality if and only if $G$ regular bipartite.
\end{theorem}
 Zhang in \cite{zhang2004}  followed Li and Pan's method and
 improved the above result.
 \begin{theorem}(\cite{zhang2004})\label{zhang2004-th3}
 Let $G$ be a simple connected graph. Then
 \begin{equation}\label{zhang2004-3}
 \lambda_1(G)\le \max\{ d(u)+\sqrt{d(u)m(u)}\ | \ u\in V(G)\}
 \end{equation}
 with equality if and only if $G$  is bipartite regular or semiregular.
  \end{theorem}
\begin{proof}
 Let $x=(x_v, v\in
 V(G))^T$ be an eigenvector with $||x||_2=1$ corresponding to
 $\lambda(G)$. Thus $L(G)x=\lambda_1(G)x$. Hence for any $u\in
 V(G)$,
 $$\lambda_1(G)x_u=d(u)x_u-\sum_{v\in V(G)}a_{uv}x_v=\sum_{(u,v)\in
 E(G)}(x_u-x_v).$$
  By the Cauchy-Schwarz inequality, we have
 \begin{eqnarray*}
 \lambda_1(G)^2x_u^2 &\le & (\sum_{(u,v)\in E(G)}1^2)( \sum_{(u,v)\in
 E(G)}(x_u-x_v)^2)\\
 &=& d(u)^2x_u^2+2d(u)x_u^2(\lambda_1(G)-d(u))+d(u)\sum_{(u,v)\in E(G)}
 x_v^2.
 \end{eqnarray*}
 Hence
\begin{eqnarray*}
 \sum_{u\in V(G)}\lambda_1(G)^2x_u^2
 &\le & \sum_{u\in
 V(G)}(2d(u)\lambda_1(G)-d(u)^2)x_u^2+\sum_{u\in
 V(G)}d(u)\sum_{(u,v)\in E(G)}x_v^2\\
&=& \sum_{u\in
 V(G)}(2d(u)\lambda_1(G)-d(u)^2)x_u^2+\sum_{u\in
 V(G)}d(u)m(u)x_u^2.
 \end{eqnarray*}
  Therefore,
   we have
   $$\sum_{u\in
   V(G)}(\lambda_1(G)^2-2d(u)\lambda_1(G)+d(u)^2-d(u)m(u))x_u^2\le
   0.$$
   Then there must exist a vertex $u$ such that
   $$\lambda_1(G)^2-2d(u)\lambda_1(G)+d(u)^2-d(u)m(u)\le
   0,$$
 which implies
$\lambda_1(G)\le d(u)+\sqrt{d(u)m(u)}.$ it follows that
 (\ref{zhang2004-3}) holds.

If  $G$ is bipartite regular or semi-regular, it is  easy  to see
that equality in (\ref{zhang2004-3}) holds by a simply
calculation.

Conversely, if equality in (\ref{zhang2004-3}) holds, it follows
from the  above     proof  that for each $u\in V(G), (u,v)\in
E(G), (u,w)\in E(G)$, we have $x_u-x_v=x_u-x_w,$ which implies
that all $x_v$ are equal for all vertices adjacent to vertex $u$.
 Fixed a vertex $w\in V(G)$, we may define that $V_1(G)=\{v\in
V(G)|$ the distance between $v$ and $w$ is even $\}$ and
$V_2(G)=\{v\in V(G)|$ the distance between $v$ and $w$ is odd
$\}$. Clearly, $V_1$ and $V_2$ are a partition of $V(G)$. Since
$G$ is connected, it is not difficult to see that all $x_v$ are
equal for any $v\in V_1$ and denoted by $a$, and that all $x_v$
are equal for any $v\in V_2$ and denoted by $b$. We claim that $G$
is bipartite. In fact, if there exists an edge $(u_1,u_2)\in
E(G)$, where $u_1,u_2\in V_1$ or $u_1,u_2\in V_2$, then $a=b$.
Hence $\lambda_1(G)x_w=\sum_{(v,w)\in E(G)}(x_w-x_v))=0$ which
implies $x_w=0$.  Therefore  $x=0$ and it is a contradiction.  For
any $u\in V_1$, we have $\lambda_1(G)x_u=\sum_{(v,u)\in
E(G)}(x_u-x_v))=(a-b)d(u)$, which result in
$d(u)=\frac{a\lambda_1(G)}{a-b}$ for any $u\in V_1$. Similarly,
$d(u)=\frac{-b\lambda_1(G)}{a-b}$ for any $u\in V_2$. Hence we
conclude that $G$ is regular or semi-regular.
\end{proof}
Since $ (d(u)+\sqrt{d(u)m(u)})\le 2d(u)(d(u)+m(u)),$ for any $u\in
V(G),$   we have that (\ref{zhang2004-3}) is always better than
(\ref{li2001-1}).

On the other hand, if the common neighbors of two adjacent
vertices are involved,  (\ref{li2001-1}) can be also improved. Das
in \cite{das2003} and \cite{das2004} showed the following
\begin{theorem}(\cite{das2003})\label{das2003-th}
Let $G$ be a   simple connected graph. Denote by
$$ m^{\prime}(u)=\frac{\sum_{v~u}(d(u)-|N(u)\bigcap N(v)|)}{d(u)},$$
where $v~u$ means that $v$ and $u$ are adjacent and $N(u)$ is the
set of all neighbor vertices of $u$. Then
\begin{equation}\label{das2003-3}
\lambda_1(G)\le \max\{\sqrt{2d(u)(d(u)+m^{\prime}(u))}\ | \ u\in
V(G)\}
\end{equation}
with equality if and only if $G$ bipartite regular.
\end{theorem}

With aid of the  relationships between the eigenvalues and
eigenvectors, we improved and generalized some equalities and
inequalities for the largest Laplacian eigenvalue. For example, in
2002, Zhang and Li \cite{zhang2002a} generalized the result for
the largest eigenvalue of mixed graphs.  In 2003, Zhang and Luo in
\cite{zhang2003c} were able to get the new upper bounds for the
Largest Laplacian eigenvalues of mixed graphs (including simple
graphs), while in 2004, Das in \cite{das2004} also obtained the
same result for simple graphs.

\begin{theorem}(\cite{das2004}, \cite{zhang2003c})
\label{zhang2003-das2004-th} Let $G$ be a simple connected graph
of order $n$. Denote by $d(u)$ and $m(u)$ the degree  and average
2-degree of the vertex $u\in V(G)$, respectively. Then
\begin{equation}\label{zhang2003-das2004-1}
\lambda_1(G)\le \displaystyle\max
\left\{\frac{d(u)+d(v)+\sqrt{(d(u)-d(v))^2+4m(u)m(v)}}{2}\ | \
(u,v)\in E(G) \right\}
\end{equation}
with equality if and only $G$ is bipartite regular or semiregular.
\end{theorem}

\subsection{ The upper bounds versus related matrices } In this subsection, we introduced
another approach to obtain the upper bound for the largest
Laplacian eigenvalue.  Li and Pan in \cite{li2004} used the
relationships of the eigenvalues of between the matrix
$K(G)=D(G)+A(G)$ and $L(G)$, and nonnegative matrix theory to
present some upper bounds for the largest Laplacian eigenvalue of
$G$.
\begin{lemma}(\cite{li2004}, \cite{liu2004})\label{li2004-lem}
Let $G$ be a simple connected graph and let $f(x)$ be a polynomial
on $x$. Denote by $\rho(K)$ the spectral radius of the matrix
$K=D(G)+A(G)$. Let $R_v(f(K))$ be the  corresponding $v-$th row
sum of $f(K)$. Then
\begin{equation}\label{li2004-1}
\min\{ R_v(f(K))\ | \ v\in V(G)\}\le f(\rho(K))\le
\max\{R_v(f(K))\ | \ v\in V(G)\}.
\end{equation}
Moreover, if the row sums of $f(K))$ are not all equal, then both
inequalities in (\ref{li2004-1}) are strict.
\end{lemma}
\begin{proof} Let $x=(x_v, v\in V(G))^T$ be a positive eigenvector
of $K$ with $sum_{v\in V(G)}x_v=1$. Then by
$$f(K)x=f(\rho(K))x,$$
we have
$$f(\rho(K))=f(\rho(K))\sum_{x\in V(G)}x_v=\sum_{v\in
V(G)}(f(K)x)_v=\sum_{v\in V(G)}x_vR_v(f(K)).$$ Therefore the
desired result holds since the entries of $x$ are positive and
their sum is  equal to 1.
\end{proof}
\begin{theorem}(\cite{li2004})\label{li2004-th}
Let $G$ be a simple connected graph with $n$ vertices and $m$
edges. Denote by $\Delta$ and $\delta$ the maximum and minimum
degrees of $G$, respectively. Then
\begin{equation}\label{li2004-2}
\lambda_1(G)\le
\frac{\delta-1+\sqrt{(\delta-1)^2+8(\Delta^2+2m-(n-1)\delta)}}{2}
\end{equation}
with equality if and only if $G$ is bipartite and regular.
\end{theorem}
\begin{proof}
Let $K=D(G)+A(G)$. Then $K^2=D(G)^2+D(G)A(G)+A(G)D(G)+A(G)^2$.
Then the $u-$row sum of $K^2$ is
\begin{eqnarray*}
R_u(K^2)&=&2d(u)+2\sum_{v~u}d(v)=2d(u)^2+4m-2d(u)-2\sum_{v\not\sim
u, v\neq u}d(v)\\
&\le &
2\Delta^2+4m-2d(u)-2(n-1-d(u))\delta\\
&=&2\Delta^2+4m+2(\delta-1)d(u)-2(n-1)\delta.
\end{eqnarray*}
Let $f(x)=x^2-(\delta-1)x$. It follows from Lemma~\ref{li2004-lem}
that
$$\rho(K)^2-(\delta-1)\rho(K)\le 2\Delta^2+4m-2(n-1)\delta.$$
Combining the above inequality and (\ref{graph-line-eig}), we are
able to obtain (\ref{li2004-2}).
\end{proof}
Using the similar method,  Li et.al in \cite{li2004} and Liu et.al
in \cite{liu2004} gave the following:
\begin{theorem}(\cite{li2004},
\cite{liu2004})\label{li2004-liu2004} Let $G$ be a simple
connected graph with $n$ vertices and $m$ edges. Denote by
$\Delta$ and $\delta$ the maximum and minimum degrees of $G$,
respectively. Then
\begin{equation}\label{li2004-liu2004-1}
\lambda_1(G)\le
\frac{\Delta+\delta-1+\sqrt{(\Delta+\delta-1)^2+8(2m-(n-1)\delta)}}{2}
\end{equation}
with equality if and only if $G$ is bipartite and regular.
\end{theorem}

\subsection{Always nontrivial upper bounds }

 In the above subsections, several kind  upper bounds for the largest
 Laplacian eigenvalue are presented. However, sometime these
 bounds exceed the number of vertices in $G$, which becomes an
 trivial upper bounds.  Rojo et.al. in \cite{rojo2000} obtained an
 always nontrivial upper bound. Their result is
 \begin{theorem}(\cite{rojo2000})\label{rojo2000-th}
 Let $G$ be a simple graph. Denote by $N(u)$ the set of all
 neighbor vertices of vertex $u$ in $G$. Then
 \begin{equation}\label{rojo20000-1}
 \lambda_1(G)\le \max\left\{ d(u)+d(v)-|N(u)\bigcap N(v)| \ |\
 u,v\in V(G)\right\}.
 \end{equation}
 \end{theorem}
Before  giving an proof,  we need the following Lemma
\begin{lemma}(\cite{bauer1969})\label{bauer-th}
Let $B=(b_{ij})$ be an $n\times n$ nonnegative matrix. Denote by
$\xi(B)$ the second largest modulus  of the eigenvalues of $B$. If
$w=(w_1,\cdots, w_n)^T$ is a positive eigenvector of $B$
corresponding to the spectral radius $\rho(B),$ then
\begin{equation}\label{bauer-1}
\xi(B)\le
\frac{1}{2}\max\left\{\sum_{k=1}^{n}w_k|\frac{b_{ik}}{w_i}-
\frac{b_{jk}}{w_j}| \ | \ 1\le i,j\le n\right\}.
\end{equation}
\end{lemma}
Now we use the lemma to prove Theorem~\ref{rojo2000-th}. Let
$B=L(G)+ee^T$, where $e$ is the all ones $n-$dimensional column
vector. Thus $M$ has a positive eigenvector $e$ corresponding to
$\rho(B)=n$ and $\xi(B)=\lambda_1(G)$. Then from
Lemma~\ref{bauer-th} and some calculations, it is not difficult to
get the desired result. Clearly  this upper bound is always
nontrivial.  But we notice that the vertices $u$ and $v$ in
Theorem~\ref{rojo2000-th} may be  or not adjacent.  It stimulated
researcher  to consider whether this result may be improved by the
adjacent relationships.  In 2003, Das \cite{das2003} improved this
upper bound. Further, Das in \cite{das2004} considered  when the
upper bound is attained and proposed a conjecture.  Yu at.el in
\cite{yu2005} confirmed the conjecture.  Before stating this
theorem, we need the following notation.  Let $F=(V, E)$ be a
semiregular with bipartition $V=V_1\bigcup V_2$ and  let
$F^+=(V,E^+)$ be a {\it super graph} of F constructed by joining
those pairs of vertices of $V_1$ (or $V_2$) which have same set of
neighbors in the other set $V_2$ (or $V_1$), if such pairs exist,
where $E^+$ is equal to E with some new edges (if new edges were
constructed).
\begin{theorem}(\cite{das2003}, \cite{das2004}, \cite{yu2005})\label{das2003-th}
Let $G$ be a simple connected graph. Then
\begin{equation}\label{das2003-1}
\lambda_1(G)\le \max\left\{ d(u)+d(v)-|N(u)\bigcap N(v)| \ |\
 (u,v)\in E(G)\right\}
 \end{equation}
with equality if and only if $G$ is a super graph of a semiregular
graph.
 \end{theorem}

\subsection{The lower  bounds for the largest Laplacian eigenvalue }
 The first lower  bound for the largest Laplacian eigenvalue may be contributed to
Fiedler \cite{fiedler1973}. He showed  the following result
\begin{theorem}\label{fiedler1973-th}(\cite{fiedler1973}) Let $G$
be a graph with on $n$ vertices and the maximum degree $\Delta$.
Then
\begin{equation}\label{fiedler1973-1}
\lambda_1(G)\ge \frac{n}{n-1}\Delta
\end{equation}
\end{theorem}

  Grone and  Merris in
 \cite{grone1994} improved (\ref{fiedler1973-1}).
  Moreover, Zhang and Luo in \cite{zhang2002b}  gave a new proof of this lower bound and
 characterized  equality situation.
 \begin{theorem}(\cite{grone1994}, \cite{zhang2002b})\label{grone1994-th}
 Let $G$ be a simple connected graph with at least one edge and
 the maximum degree $\Delta$. Then
 \begin{equation}\label{grone1994-11}
 \lambda_1(G)\ge \Delta+1
 \end{equation}
 with equality if and only if  there exists a vertex is adjacent all other vertices in $G$.
  \end{theorem}
\begin{proof}
 It is easy to see that $G$ contains a star graph $H$ with $\Delta+1$
vertices. By a simple calculation, the largest Laplacian
eigenvalue of $H$ is $\Delta+1$. Hence the result follows from
Theorem~4.1 in \cite{grone1990b}.

If there exists a vertex is adjacent all other vertices in $G$,
then $\Delta=n-1,$ where $n$ is the number of vertices in $G$.  By
(\ref{grone1994-11}) and Lemma~\ref{graph-line}, equality in
(\ref{grone1994-11}) holds. Conversely, if $\Delta<n-1$, then let
$d(z)=\Delta$ and there exist vertices $y_1$ and $y_2$ such that
$(z, y_1)\in E(G), (z,y_2)\notin E(G)$ and $(y_1,y_2)\in E(G)$.
Let $H_1$ be a subgraph of $G$ obtained from a star graph with
$\Delta+1$ vertices and joining a  new vertex and new edge. By a
simple calculation and Theorem~4.1 in \cite{grone1990b},
$\lambda_1(G)\ge \lambda_1(H_1)>\Delta+1$.
\end{proof}

Another lower bound for the largest Laplacian eigenvalue in terms
of the number of vertices and edges  was given in
\cite{zhang2001}.
\begin{theorem}(\cite{zhang2001})\label{zhang2001-th}
Let $G$ be a simple graph with $n$ vertices and $m$ edges. Then
\begin{equation}\label{zhang2001-1}
\lambda_1(G)\ge \frac{1}{n-1}\left(
2m+\sqrt{\frac{2m(n(n-1)-2m)}{n(n-2)}}\ \ \right)
\end{equation}
with equality if and only if $G$ is the complete graph.
\end{theorem}
\begin{proof}
 Clearly,
 $$((n-1)\lambda_1-Tr(L(G)))^2\ge
 \sum_{i=1}^{n-1}(\lambda_1-\lambda_i)^2,$$
while
$$\sum_{i=1}^{n-1}(\lambda_1-\lambda_i)^2=Tr(L(G)^2)-2\lambda_1Tr(L(G))+(n-1)\lambda_1^2.$$
Since $Tr(L(G))=2m$  and $Tr(L(G)^2)\ge 2m+\frac{(2m)^2}{n}$, we
have
$$((n-1)\lambda_1-2m)^2\ge
(2m+\frac{(2m)^2}{n})-4m\lambda_1+(n-1)\lambda_1^2.$$
 By solving this quadratic form, it is easy to obtain
 (\ref{zhang2001-1}).
\end{proof}

 Das in \cite{das2004b} considered the largest Laplacian
 eigenvalues of special subgraphs of a graph and obtained  a lower
 bound for the largest Laplacian eigenvalue of graphs in term of
 degree sequence and their neighbor sets.
 \begin{theorem}(\cite{das2004b})\label{das2004-th1}
Let $G$ be a simple graph with at least one edge.  Denote by
$c_{uv}=d(u)-|N(u)\bigcap N(v)|-1$, $t_u=d(u)^2+2d(u)$,  Then
\begin{equation}\label{das2004b-1}
\lambda_1(G)\ge \max\left\{\sqrt{\frac{1}{2}
\left(t_u-2d(v)-2+\sqrt{(t_u+2d(v)+4)^2+4c_{uv}c_{vu}}\ \right)\
 }\ \right\},
\end{equation}
where the maximum is taken over all pairs $(u,v)\in E(G)$.
\end{theorem}

\subsection{The  upper and lower bounds for special graphs}

Now we turn to consider the upper and lower bounds for the largest
Laplacian eigenvalue of special graphs. Zhang and Luo in
\cite{zhang2002b}  provided the following lower bound for the
largest Laplacian eigenvalue of triangle-free graphs.
\begin{theorem}(\cite{zhang2002b})\label{zhang2002b-th}
Let $G=(V,E)$ be a triangle-free graph. If $d_u$ and $m_u$ are the
degree and the average 2-degree of a vertex $u$, respectively,
then
\begin{equation}\label{zhang2002b-1}
\lambda_1(G)\ge
\displaystyle\max\{\frac{1}{2}(d(u)+m(u)+\sqrt{(d(u)-m(u))^2+4d(u)},\
\ u\in V\}.
\end{equation}
\end{theorem}
\begin{proof} Let $L(U)$ be the principal submatrix of $L(G)$
corresponding to $U$, where $U=\{ u, v_1,\cdots, v_{k}\}$ is the
closed neighborhood of a vertex $u$ and $d(u)=k$. Obviously,
$\lambda_1(L(G))\ge \lambda_1(L(U))$. Since $G$ is triangle-free,
we may assume that

\[
L(U)=\left(
\begin{array}{rrrrr}
d_u&-1&-1&\cdots&-1\\
-1&d_{v_1}&0&\cdots& 0\\
\cdots&\cdots&\cdots&\cdots&\cdots\\
-1&0&0&\cdots&d_{v_{k}}
        \end{array}\right).\]
With  elementary calculations, we have that the characteristic
polynomial of $L(U)$ is

$$ det(\lambda I-L(U))=(\lambda-d(u)-\sum_{i=1}^{k}\frac{1}{\lambda-d(v_i)}) \prod_{i=1}^{k}
(\lambda-d(v_i)).$$
Note that $\lambda_1(L(G))\ge
\lambda_1(L(U))>d({v_i})$ for each $i=1,\cdots,k$. Hence
$\lambda_1(L(G)) $ satisfies

$$ \lambda_1(L(G))-d(u)\ge \sum_{i=1}^{k}\frac{1}{\lambda_1(L(G))-d({v_i})}.$$
By Cauchy-Schwarz inequality, we have
$$\sum_{i=1}^{k}(\lambda_1(L(G))-d({v_i}))\sum_{i=1}^{k}\frac{1}{\lambda_1(L(G))-d({v_i})}\ge
\left(\sum_{i=1}^{k}\frac{\sqrt{\lambda_1(L(G))-d({v_i})}}{\sqrt{\lambda_1(L(G))-d({v_i})}}\right)^2=
k^2.$$
Hence

$$
\lambda_1(L(G))-d(u)\ge
\frac{k^2}{\sum_{i=1}^{k}(\lambda_1(L(G))-d({v_i}))}=
\frac{d(u)}{\lambda_1(L(G))-m(u)}\ ,
$$
since $m(u)=\frac{1}{k}\sum_{i=1}^k d({v_i}).$ This inequality
yields the desired result.
\end{proof}

 Yu et al. in \cite{yu2004} used the 2-degree vertex to present a
 lower bound for the Laplacian eigenvalue of bipartite graphs.
\begin{theorem}(\cite{yu2004})\label{yu2004-th}
Let $G$ be a simple connected bipartite graph. Then
\begin{equation}\label{yu2004-1}
\lambda_1(G)\ge\sqrt{\frac{\sum_{v\in
V(G)}d(v)^2(d(v)+m(v))^2}{\sum_{v\in V(G)}d(v)^2}}
\end{equation}
with equality if and only if $G$ is regular or semiregular.
\end{theorem}

Hong and Zhang in \cite{hong2005} gave another lower bound for the
largest Laplacian eigenvalue of bipartite graphs.

\begin{theorem}( \cite{hong2005})\label{hong2005-th}
Let $G$ be a simple connected bipartite graph. Then
\begin{equation}\label{hong2005-1}
\lambda_1(G) \ge 2 + \sqrt{\frac{1}{m}\sum_{u \sim v}(d(u) + d(v)
- 2)^2}\ , \end{equation}
 where $m$ is the edge number of $G$. Moreover,
equality in (\ref{hong2005-1}) holds if and only if $G$ is either
a regular connected bipartite graph, or a semiregular connected
bipartite graph, or the path with four vertices.
\end{theorem}

 If we consider tree, what are about  upper and lower bounds for
 the largest Laplacian eigenvalue?
 Stevanovi${\rm \acute{c}}$  in \cite{Stevanovic2003} presented an
upper bound for the largest Laplacian eigenvalue of a tree in
terms of the largest vertex degree.
\begin{theorem}(\cite{Stevanovic2003})\label{Stevanovic2003-th}
 Let $T$ be a tree with
the largest vertex degree $\Delta$. Then
\begin{equation}\label{Stevanovic2003-1}
\lambda(T)<\Delta+2\sqrt{\Delta-1}.
\end{equation}
\end{theorem}
In 2005, Rojo \cite{rojo2005} improved  Stevanovi${\rm
\acute{c}}$'s result.
\begin{theorem}(\cite{rojo2005})\label{rojo2005-th}
Let $T$ be a tree with the largest vertex degree $\Delta$. Let $u$
be a vertex of $T$ with $d(u)=\Delta$. Denote by $k-1$ the largest
distance from $u$ to any other vertex of tree. For $j=1,\cdots,
k-1$, let $\delta_j=\max\{d(v): dist(v,u)=j\}.$ Then
\begin{equation}\label{rojo2005-1}
\lambda(G)<\max\{\max_{2\le j\le
k-2}\{\sqrt{\delta_j-1}+\delta_j+\sqrt{\delta_{j-1}-1}\},
\sqrt{\delta_1-1}+\delta_1+\sqrt{\Delta}, \Delta+\sqrt{\Delta}\}.
\end{equation}
\end{theorem}
From the proofs of \cite{Stevanovic2003} and \cite{rojo2005}, we
are able to this upper bound is not achieved. It is natural to ask
what is the best upper bound for trees. Thus we may propose the
following question:
\begin{question}\label{stevanovic2003-ques}
Let $T$ be a tree with the largest vertex degree $\Delta$. What is
the best upper bound for the largest Laplacian eigenvalue of $T$?
\end{question}

\subsection{The bounds in terms of  graph parameters}
In the above several sections, we have mainly  investigated some
upper and lower bounds for the largest eigenvalue of graphs in
terms of the following basic invariants of $G$, including,  the
vertex number, the edge number, the maximum and minimum degrees,
2-average degree, degree sequence.  In this subsection, we just
focus on relations between the largest Laplacian eigenvalue and
other graphs parameters.

 A subset $U$ of  vertex set $V$ of a graph $G=(V,
 E)$ is called {\it an independent set } of $G$ if no two vertices
 of $U$ are adjacent in $G$.  The {\it independence number } $\alpha(G)$ of
 $G$  is the maximum size of independent sets of $G$.
 In 2004, Zhang  \cite{zhang2004b} proved two conjectures on the
 Laplacian eigenvalue and the independence number.

 \begin{theorem}(\cite{zhang2004b})\label{zhang2004b-th}
 Let $G$  be a graph of order $n$ with at least one edge and the
 independence number $\alpha(G)$. Then
 \begin{equation}\label{zhang2004b-1}
 \lambda_1(G)\ge \frac{n}{\alpha(G)}
 \end{equation}
 with equality if and only if $\alpha(G)$ is a factor and $G$ has
 $\alpha(G)$ components each of which is the complete graph
 $K_{\frac{n}{\alpha(G)}}.$
\end{theorem}
 In 2005, Lu et al. \cite{lu2005} also obtained the same result for connected graphs.
  Recently, Nikiforov in \cite{Nikiforov2007} gave a slight
 improvement and showed that
 $\lambda_1(G)\ge\left\lceil\frac{n}{\alpha(G)}\right\rceil$, where
  $\lceil x\rceil$ the smallest integer no less than
 $x$.
Let  $K_{1,m}$ denote the star  on $m+1$
  vertices. If  $\frac{n-1}{2}<m\le n-1$, then $T_{n,m}$ is the tree created
  from $K_{1,m}$ by adding a pendent edge to $n-m-1$ of the
  pendent vertices of $K_{1, m}$.

\begin{theorem}(\cite{zhang2004b})
\label{zhang2004b-th3}
 Let $T$ be  a tree of order  $ n$  and the
independence number $\alpha(T)$. Denote by  $a$ the largest root
of the equation $x^3-(\alpha(T)+4)x^2+(3\alpha(T)+4)x-n=0$. Then
\begin{equation}\label{zhang2004b-3}
\lambda_1(T)\le a
\end{equation}
with  equality if and only if $T$ is $T_{n, \alpha(T)}$.
\end{theorem}

 A {\it matching} in a simple graph $G$ is a set of edges with no shared common vertex
The {\it matching number } of $G$ is the maximum size  among all
matching in $G$.
 Guo in \cite{guo2003} showed that the largest Laplacian
 eigenvalue of  a tree in terms of the matching number.
 \begin{theorem}
 (\cite{guo2003})\label{guo2003-th}
 Let $T$ be a tree of order $n$ with the matching number
 $\beta(T)$.
Denote by  $a$ the largest root of the equation
$x^3-(n-\beta(T)+4)x^2+(3n-3\beta(T)+4)x-n=0$.
  Then
\begin{equation}\label{guo2003-1}
\lambda_1(T)\le a
\end{equation}
with  equality if and only if $T$ is $T_{n, n-\beta(T)}$.
\end{theorem}

Let $G$ be a simple graph and let $H$ be any bipartite subgraph of
$G$ with the maximum edges. Thus
$$b(G)=\frac{|E(H)|}{|E(G)|}$$
is called {\it the bipartite density} of $G$.  Berman and Zhang in
\cite{berman2003} gave a lower bound  for the largest Laplacian
eigenvalue of cubic graphs in terms of their bipartite density.
Stevanovi\'{c} in \cite{Stevanovic2004} characterized all exemtral
graphs which attain the lower bound.

\begin{theorem}(\cite{berman2003},
\cite{Stevanovic2004})\label{berman2003-th} Let $G$ be a connected
cubic graph of order $n$ with the bipartite density $b(G)$. Then
\begin{equation}\label{berman2003-1}
\lambda_1(G)\ge \frac{10b(G)-4}{b(G)}
\end{equation}
with equality if and only if $G$ is bipartite graph, or the
complete graph $K_4$, or the Petersen graph, or the four special
graphs  of order 10.
\end{theorem}

\section{The second largest Laplacian eigenvalue}
Since there are a lot of upper and lower bounds for the largest
Laplacian eigenvalues of graphs, On upper and lower bounds for the
second largest Laplacian eigenvalue of graphs, what can we say? Up
to now, there are just a few results on it.
 Firstly, Zhang and Li in \cite{zhang1998} investigated the second
 largest Laplacian eigenvalue of a tree.   They obtained the upper
 bound in terms of the number of vertices and characterized all
 extremal graphs which attained the upper bound.
 \begin{theorem}(\cite{zhang1998})\label{zhang1998-th}
 Let $T$ be a tree of order $n$. Denote by $\lceil x\rceil$ the smallest integer no less than
 $x$. Then
 \begin{equation}\label{zhang1998-1}
 \lambda_2(T)\le\left \lceil\ \frac{n}{2}\ \right\rceil
 \end{equation}
 with equality if and only if $n$ is even and $T$ is obtained  joining one edge from any one
 vertex  to another vertex between the two copies star graphs $K_{1,\frac{n}{2}-1}$.
\end{theorem}

Using the relations between graph partition and the Laplacian
eigenvalue and Cauch-Poincare separation theorem, Li and Pan in
\cite{li2000} showed the the second largest Laplacian eigenvalue
of a graph is at least its second largest degree.

\begin{theorem}(\cite{li2000})\label{li2000-th}
Let $G$ be a simple connected graph with $n\ge 3$ vertices. Denote
by $d_2$ the second largest degree of $G$. Then
\begin{equation}\label{li2000-11}
\lambda_2(G)\ge d_2
\end{equation}
with equality if $G$ is a complete bipartite graph.
 \end{theorem}

Das in \cite{das2004b} studied the Laplacian eigenvalues of
induced subgraph of a graph obtained from the vertices of two
vertices with the largest two degrees and their neighbors. Basing
these properties and Cauch-Poincare separation theorem,  He
improved Li and Pan's lower bound.
\begin{theorem}(\cite{das2004b})\label{das2004-th2}
Let $G$ be a simple connected graph with at least three vertices.
Denote by $d_1=d(u)$ and $d_2=d(v)$ the largest and second largest
degree of $G$, respectively, and $c_{uv}=|N(u)\cap N(v)|$. Then
\begin{equation}\label{das2004b-2}
\lambda_2(G)\ge \left\{\begin{array}{ll}
\frac{d_2+2+\sqrt{(d_2-2)^2+4c_{uv}}}{2}, &\,  {\rm if }\ (u,v)\in E\\
\frac{d_2+1+\sqrt{(d_2+1)^2-4c_{uv}}}{2}, &\, {\rm if }\
(u,v)\notin E.
\end{array}\right.
\end{equation}
\end{theorem}
For most upper and lower bounds  for the largest Laplacian
eigenvalues, we are able to characterize all extremal graphs which
attain their bounds. For the same season, we also expect to
characterize all extremal graphs which achieve this lower bounds.
Although (\ref{das2004b-2}) is better than (\ref{li2000-11}),  it
is  still not able to help us to find all extremal graphs  which
attain the lower bound (\ref{li2000-11}). Pan and Hou in
\cite{pan2003} gave the two necessary conditions for graphs with
the second largest Laplacian eigenvalue equal to the second
largest degree.
\begin{theorem}(\cite{pan2003})\label{pan2003-th}
Let $G$ be a simple connected graph of order $n\ge 3$ other than
the star graph. Denote by $d_1=d(u)$ and $d_2=d(v)$ the largest
and second largest degree of $G$, respectively. Assume that
$\lambda_2(G)=d_2$.

(1) If $(u,v)\in E(G)$, then $N(u)=N(v)$.

(2) If $(u,v)\notin E(G),$ then $N(u)\cap N(v)=\empty$
 $d_1=d_2=\frac{n}{2}$.
\end{theorem}
On the other hand, there are many graphs whose second largest
Laplacian eigenvalue is equal to its second largest degree, for
example, double star graphs which is obtained from joining a new
edge from the centers of two star graphs, etc. Basing the above
situation, Li {\it et al.} in \cite{li2003} proposed the following
question:
\begin{question}(\cite{li2003})\label{li20003-th}
Characterize all extremal graphs  such that its second largest
Laplacian eigenvalue is equal to its second largest degree.
\end{question}

\section{The $k-$th largest Laplacian eigenvalue}
In this section, we consider some upper and lower bounds for the
$k-$th largest Laplacian eigenvalues of graphs or trees.
  Zhang and Li in \cite{zhang2001} gave the upper and lower bounds
  for the $k-$th largest Laplacian eigenvalues of graphs in terms
  of the number of vertices,  edges and the number of spanning trees.
  \begin{theorem}(\cite{zhang2001})\label{zhang2001-th2}
  Let $G$ be a simple connected graph of order $n$ with $m$ edges. Denote by
  $M(G)=\min\{m((n-4)m+2(n-1)), 2m(n(n-1)-2m)\}.$ Then for
  $k=1,\cdots, n-1$,
  \begin{equation}\label{zhang2001-2}
  \lambda_k(G)\le
  \frac{1}{n-1}\left\{2m+\sqrt{\frac{n-k-1}{k}M(G)}\ \right\}
  \end{equation}
with equality in (\ref{zhang2001-2}) for some $1\le k_0\le n-1$ if
 and only if $G$ is the complete graph or star graph.
\end{theorem}
\begin{proof}
Clearly,
$$Tr(L(G)^2)=\sum_{i=1}^k\lambda_i^2+\sum_{i=k+1}^{n-1}\lambda_i\ge\frac{(\sum_{i=1}^k
\lambda_i)^2}{k}+\frac{(\sum_{i=k+1}^{n-1}\lambda_i)^2}{n-k-1}.$$
Let $\varphi_k=\sum_{i=1}^k\lambda_i$. Then
$$Tr(L(G)^2)\ge
\frac{\varphi_k^2}{k}+\frac{(2m-\varphi_k)^2}{n-k-1}$$ which
implies
$$\lambda_k\le \frac{\varphi_k}{k}\le
\frac{1}{n-1}\left\{2m+\sqrt{\frac{n-k-1}{k}[(n-1)Tr(L(G)^2)-4m^2]}\right\}.$$
 We observe that
 $$(n-1)Tr(L(G)^2)-4m^2=(n-1)\sum_{v\in V}d(v)^2+2m(n-1)-4m^2\le
 m((n-4)m+2(n-1))$$
 and
 $$(n-1)Tr(L(G)^2)-4m^2\le (n-1)\sum_{v\in
 V}d(v)(n-1)+2m(n-1)-4m^2=2m(n(n-1)-2m),$$
 since $d(v)\le n-1$.
  Hence (\ref{zhang2001-2}) holds.
  \end{proof}
Next, Zhang and Li in \cite{zhang2001} used the number of spanning
trees  and edges to obtain the lower bounds for the $k-$th largest
Laplacian eigenvalues of graphs.

\begin{theorem}(\cite{zhang2001}) \label{zhang2001-th3}
Let $G$ be a simple connected graph of order $n$ with $m$ edges.
Denote by $\tau$ the number of spanning trees of $G.$ Then
\begin{equation}\label{zhang2001-3}
\lambda_k\ge
\frac{1}{n-k}\left\{(n-1)(2^{n-k}n\tau)^{\frac{1}{n-1}}-2m\right\}.
\end{equation}
If $G$ is a strongly regular graph on the parameters $(a^2,
2(a-1), a-2,2)$, equality in (\ref{zhang2001-3}) holds for
$k=(a-1)^2+1$.
\end{theorem}

 From \cite{grone1994} and \cite{li2000}, we have  that
 $\lambda_1(G)\ge d_1+1$ and $\lambda_2(G)\ge d_2$ if $d_1\ge
  d_2\ge\cdots\ge d_n$ are degree sequence of $G$.  These results
 arise a question what about the relations between the $k$-th
 largest Laplacian eigenvalue and the $k-$th largest degree.
 Guo in \cite{guo2007b} found that in general $\lambda_k\ge d_k$
 does not hold.  But he showed that the following inequality.
 \begin{theorem}(\cite{guo2007b})\label{guo2007b-th}
 Let $G$ be a simple connected graph with at least four vertices.
 Denote by $d_3$ the third largest degree of $G$.
 Then
 \begin{equation}\label{guo2007b-1}
 \lambda_3(G)\ge d_3-1.
 \end{equation}
 \end{theorem}
 Basing his result and observing, he proposed the following
 conjecture:
 \begin{conjecture}(\cite{guo2007b})\label{guo2007b-con}
 Let $G$ be a simple connected graph  of order $n$. Denote by
 $d_k$ the $k-$th largest Laplacian eigenvalue of $G.$
 Then
 \begin{equation}\label{guo2007b-2}
 \lambda_k(G)\ge d_k-k+2,\ \ {\rm for \  all}\  k=1,\cdots, n-1.
 \end{equation}
 \end{conjecture}
Recently,  Wang {\rm et al.} in \cite{wang2008} confirmed this
conjecture and characterized all extremal graphs which attain the
lower bounds.

From (\ref{zhang1998-1}),  we may obtain  $\lambda_k(G)\le
\lceil\frac{n}{k}\rceil$ for a tree of order if $k=1,2. $ It is
natural to expect whether the result is able to generalize for any
$k$.  Recently, Guo \cite{guo2007a} followed this idea and showed
the following:
\begin{theorem}(\cite{guo2007a})\label{guo2007-th}
Let $T$ be a  tree of order $n$. Then
\begin{equation}\label{guo2007a-1}
\lambda_k(T)\le \left \lceil \frac{n}{k}\right \rceil \ \ {\rm for
}\ 1\le k\le n-1
\end{equation}
with equality if and only if $k|n$ and  $T$ is spanned by $k$
vertices disjoint copies of the star graph $K_{1, \frac{n}{k}-1}$.
\end{theorem}

\section{The second smallest Laplacian eigenvalue}

In 1973, Fiedler in \cite{fiedler1973} called the second smallest
Laplacian eigenvalue the algebraic connectivity of a graph, since
it is a good parameter to measure, to a certain extent, how well a
graph is connected. For example,  The second smallest eigenvalue
is positive if and only if $G$ is connected. Moreover, the
eigenvectors corresponding to the algebraic connectivity  are
called {\it Fiedler vectors} (see,\cite{fiedler1975},
\cite{kirkland2002}, \cite{kirkland1997}).  Recently, there is an
excellent survey on algebraic connectivity of graphs written by de
Abreu \cite{abreu2007}.  One of the earliest result may be is due
to Fielder \cite{fiedler1973}
\begin{theorem}(\cite{fiedler1973})\label{fiedler1973-th}
Let $G$ be a simple graph of order $n$  other than a complete
graph with vertex connectivity $\kappa(G)$ and edge connectivity
$\kappa^{\prime}(G)$. Then
\begin{equation}\label{fiedler1973-3}
2\kappa^{\prime}(G)(1-\cos(\pi/n)\le \lambda_{n-1}(G)\le
\kappa(G)\le \kappa^{\prime}(G).
\end{equation}
\end{theorem}
It is natural to investigate all extremal graphs which attain the
bound in (\ref{fiedler1973-3}). In order to characterize all
extremal graphs with $\lambda_{n-1}(G)\le \kappa(G)\le
\kappa^{\prime}(G)$, we recall the definitions of  the union and
join of graphs.  Let $G_1=(V_1, E_1)$ and $G_2=(V_2, E_2)$ be two
disjoint graphs. The {\it union} of $G_1$ and $G_2$ is
$G_1+G_1=(V_1\cup V_2, E_1\cup E_2)$ and the {\it join} $G_1\vee
G_2$ of $G_1$ and $G_2$ is a graph from $G_1+G_2$ by adding new
edges from each vertex in $G_1$ to every vertex of $G_2$.
 Kirkland et al. \cite{kirkland2002} obtained the necessary and sufficient conditions
 for the second smallest
eigenvalue equal to the vertex connectivity.
\begin{theorem}(\cite{kirkland2002})\label{kirkland2002-1}
Let $G$ be a simple connected graph or order $n$ rather than a
complete graph. Then $\lambda_{n-1}(G)=\kappa(G)$ if and only if
$G$ can be written as $G_1\vee G_2$, where $G_1$ is a disconnected
graph of order $n-\kappa(G)$ and $G_2$ is a graph of order
$\kappa(G)$ with $\lambda_{\kappa(G)-1}(G_2)\ge 2\kappa(G)-n$.
\end{theorem}

Now we present some new results which are not appeared in
\cite{abreu2007}. A {\it dominating set} in $G$ is a subset $U$ of
$V(G)$ such that each vertex in $V(G)-U$ is adjacent to at least
one vertex of $U$. The {\it domination number} $\gamma(G)$ is the
minimum size  of a dominating set in $G$. Lu et al. in
\cite{lu2005} gave an upper bound for the second smallest
Laplacian eigenvalue in terms of the domination number.
\begin{theorem}
(\cite{lu2005})\label{lu2005-th2} Let $G$ be a simple connected
graph of order $n\ge 2$. Then
\begin{equation}\label{lu2005-2}
\lambda_{n-1}(G)\le \frac{n(n-2\gamma(G)+1)}{n-\gamma(G)}
\end{equation}
with equality if and only if $G$ is the complete bipartite graph
$K_{2,2}.$
\end{theorem}

Recently, Nikiforov in \cite{Nikiforov2007} gave another upper
bound.
\begin{theorem}(\cite{Nikiforov2007})(\label{nikiforov2007-th}
Let $G$ be a simple connected graph  other than a complete graph.
Then
\begin{equation}\label{nikiforov2007-1}
\lambda_{n-1}(G)\le n-\gamma(G).
\end{equation}
\end{theorem}

We notice that (\ref{lu2005-2}) and (\ref{nikiforov2007-1}) are
not comparable.  Another important graph parameter is diameter.
There are several results on the upper and lower bounds for the
second smallest Laplacian eigenvalue in terms of diameter of $G$.
The reader may refer to \cite{abreu2007}. In here, we only present
an up-to-date result by Lu et al. \cite{lu2007}.
\begin{theorem}(\cite{lu2007})\label{lu2007-th}
Let $G$ be a  simple connected graph of order $n$ with $m$ edges
and diameter $diam(G)$. Then
\begin{equation}\label{lu2007-1}
\lambda_{n-1}(G)\ge \frac{2n}{2+n(n-1)(diam(G))-2m( diam(G))}
\end{equation}
with equality if and only if $G$ is a path of order $3$ or a
complete graph.
\end{theorem}

 For trees, we gave an upper bound for the second smallest
 Laplacian eigenvalue in terms of the independence number
 $\alpha(G)$.
 Zhang in \cite{zhang2004b} proved the following:
 \begin{theorem}(\cite{zhang2004b})\label{zhang2004b-th2}
 Let $T$ be a tree of order $n$  with
 the independence number $\alpha(T)$.
 If $T$ is not the star graph $K_{1,n-1}$ or $T_{n, n-2}$, then
 \begin{equation}\label{zhang2004b-2}
\lambda_{n-1}(T)\le \frac{3-\sqrt{5}}{2}
\end{equation}
with if and only if $T$ is $T_{n, \alpha(T)}$
   \end{theorem}
By a simple calculation,  we  have following corollary due to
Grone et al. \cite{grone1990b}
\begin{corollary}(\cite{grone1990b})
\label{grone1990b-co} Let $T$ be a tree of order $n\ge 6$  other
than the star graph $K_{1,n-1}$. Then $\lambda_{n-1}(T)<0.49.$
\end{corollary}

Merris in \cite{Merris1997} introduced the {\it doubly stochastic
matrix} of a graph which is defined to be
$\Omega(G)=(\omega_{ij})=(I+L(G))^{-1}$. Denote by
$\omega(G)=\min\{\omega_{ij}\ | \  1\le i,j\le n\}$. In the study
of relations between smallest entry of this doubly stochastic
matrix and the algebraic connectivity.  In 1998, Merris
\cite{Merris1998b} proposed the following two conjectures.

\begin{conjecture}(\cite{Merris1998b})\label{merris1998b-con1}
 Let $G$ be a graph on $n$ vertices. Then
\begin{equation}\label{merris1998b-1}
\lambda_{n-1}(G)\ge 2(n+1)\omega(G).
\end{equation}
\end{conjecture}

\begin{conjecture}(\cite{Merris1998b})
\label{merris1998b-con2} Let $E_n$ be the degree anti-regular
graph, that is, the unique connected graph whose vertex degrees
attain all values between 1 and $n-1$. Then
\begin{equation}\label{meris1998b-2}
\omega(E_n)=\frac{1}{2(n+1)}. \end{equation}
\end{conjecture}
In 2000, Berman and Zhang \cite{berman2000} confirmed Conjecture~
 \ref{merris1998b-con2}.  Recently, Zhang and Wu in
 \cite{zhang2005} firstly  obtained sharp upper and lower bounds for the smallest entries
  of doubly stochastic matrices of trees, which is used to
  disprove Conjecture~\ref{merris1998b-con1}. Hence we may propose
  the following question:
  \begin{question}\label{new-2}
  What is the best lower bound for the algebraic connectivity in
  terms of the vertex number and the smallest entry of the doubly
  stochastic matrix of a graph?
  \end{question}

\section{The sum of the Laplacian eigenvalues }
  Before presenting some results, we need to recall some
  notations.
  If $(a)=(a_1,a_2,\cdots,a_r)$ and $(b)=(b_1,b_2,\cdots, b_s)$
  are nonincreasing sequences of real  number, then $(a)$ {\it
  majorizes } $ (b)$, denoted by $(a)\succeq (b)$, if
  $$\sum_{i=1}^ka_i\ge \sum_{i=1}^kb_i, \ \ {\rm for} \ k=1,2,\cdots,
  \min\{ r, s\}
  $$ and
  $$\sum_{i=1}^ra_i=\sum_{i=1}^sb_i.$$
Moreover, if $(a)$ is a integer nonincreasing sequence, denote by
$(a)^*=(a_1^*,a_2^*,\cdots, a_t^*)$ the {\it conjugate} sequence
of $(a)$, where $a_i$ is the cardinality of the set $\{j\ | \
a_j\ge i\}$.

Since $L(G)$ is  positive semidefinite, it follows from Schur's
theorem (see \cite{marshal1979}) that   the Laplacian eigenvalues
of a graph majorizes the degree sequence (when  both are arranged
in nonincreasing order). It is not surprising that such a result
should be, to some extent, improved upon restriction to the class
of the Laplacian matrices.   Grone and Merris  in \cite{grone1994}
proposed the following two conjectures on the Laplacian
eigenvalues.

\begin{conjecture}(\cite{grone1994})\label{grone1994-con}
Let $G$ be a connected graph of order $n\ge 2$ with nonincreasing
degree sequence $(d_1,d_2,\cdots, d_n)$. Then
\begin{equation}\label{grone1994-1}
(\lambda_1(G), \lambda_2(G),\cdots, \lambda_{n-1}(G))\succeq
(d_1+1,d_2,\cdots, d_n-1). \end{equation}
\end{conjecture}

\begin{conjecture}(\cite{grone1994})\label{grone1994-con2}
Let $G$ be a connected graph of order $n\ge 2$ with nonincreasing
degree sequence $(d_1,d_2,\cdots, d_n)$. Then
\begin{equation}\label{grone1994-2}
(\lambda_1(G), \lambda_2(G),\cdots, \lambda_{n-1}(G))\preceq
(d_1^*,d_2^*,\cdots, d_n^*). \end{equation}
\end{conjecture}

On Conjecture \ref{grone1994-con}, Grone and Merris in
\cite{grone1994} showed the part result on this conjecture.

\begin{theorem}(\cite{grone1994})\label{grone1994-th}
Let $G=(V,E)$ be a  connected  graph of order $n>2$. If the
induced subgraph by subset $U$ of $V$ with $|U|=k$ contains $r$
pair disjoint edges, then
\begin{equation}\label{grone1994}
\sum_{i=1}^k\lambda_i(G)\ge\sum_{u\in U}d(u)+k-r
\end{equation}
\end{theorem}
Further,  using  M-matrix theory and graph structure,  Grone in
\cite{grone1995} confirmed Conjecture~\ref{grone1994-con}.
However,  it seems to be difficult to prove
Conjecture~\ref{grone1994-con2}. In \cite{grone1994}, Grone and
Merris only showed that $\lambda_{n-1}(G)\ge d_{n-1}^*$, in other
words, the  first and last inequalities in the majorization
inequality hold. In 2002, Duval and Reiner \cite{duval2002}
investigated the combinatorial Laplace operators associated to the
boundary maps in a shifted simplicial complex. They proposed a
generalization of Conjecture~\ref{grone1994-con2} and only proved
the following :

\begin{theorem}(\cite{duval2002})\label{duval2002-th}
Let $G$ be a connected graph with the nonincreasing degree
sequence $(d_1^*, \cdots, d_n^*)$ . Then
\begin{equation}\label{duval2002-1}
\lambda_1(G)+\lambda_2(G)\le d_1^*+d_2^*
\end{equation}
\end{theorem}
 Moreover, there are more and more evidence to indicate that
 Conjecture~\ref{grone1994-con2} may hold.  For example, Merris in
 \cite{merris1994b} studied  the relations between spectra and structure for a class of graphs which are called
  {\it degree maximal graphs} and found that $(\lambda_1(G),\cdots,
 \lambda_n(G))=(d_1^*,\cdots, d_n^*)$.  In other words, equality
 in Conjecture~\ref{grone1994-con2} holds.  Hammer and Klemans in
 \cite{hammer1996} investigated the question of which graphs have
 integer spectra and found that the
{\it threshold graphs} are Laplacian integer. In fact, the degree
maximal graphs are exactly the threshold graphs. It is known that
Conjecture~\ref{grone1994-con2} holds for regular graphs and
nearly regular graphs whose vertices have degree either $k$ or
$k-1$. In 2004, Stephen \cite{stephen2004} showed that
Conjecture~\ref{grone1994-con2} holds for trees. However, up to
now,  this Conjecture has still not been proved or disproved.

\end{document}